\documentclass[11pt,reqno]{amsart}

\usepackage[T1]{fontenc}

\usepackage{amsmath,xypic,tikz-cd}								
\usepackage{amssymb}
\usepackage{amsthm}
\usepackage{amscd}
\usepackage{amsfonts}
\usepackage{dsfont}
\usepackage{xfrac}
\usepackage{bm}
\usepackage{stmaryrd}
\usepackage[all]{xy}
\usepackage{wasysym}
\usepackage{euler} 
 
\usepackage{extarrows}

\usepackage[colorlinks, linktocpage, citecolor = purple, linkcolor = purple]{hyperref}
\usepackage{color}

\usepackage{tikz}									
\usetikzlibrary{matrix}
\usetikzlibrary{patterns}
\usetikzlibrary{matrix}
\usetikzlibrary{positioning}
\usetikzlibrary{decorations.pathmorphing}

\usepackage{enumerate}

\newtheorem{theorem}{Theorem}[section]
\newtheorem{lemma}[theorem]{Lemma}

\newtheorem{corollary}[theorem]{Corollary} 
\newtheorem{conjecture}[theorem]{Conjecture} 

\newtheorem*{thmf}{Theorem \ref{thm:hassettcycle}}
\newtheorem*{lm2}{Lemma \ref{lem:expflow}}
\newtheorem*{thm3}{Theorem \ref{thm:wittentodiag}}
\newtheorem*{thm4}{Theorem \ref{thm:wittentohassettcyc}}

\theoremstyle{definition}
\newtheorem{definition}[theorem]{Definition}
\newtheorem{example}[theorem]{Example}

\newtheorem{remark}[theorem]{Remark}

\newcommand{\vance}[1]{{\color{blue} \sf  Vance: [#1]}}
\newcommand{\renzo}[1]{{\color{red} \sf  Renzo: [#1]}}

\renewcommand{\hat}[1]{\widehat{#1}}
\renewcommand{\tilde}[1]{\widetilde{#1}}

\newcommand{\emphbf}[1]{{\bf  #1}}

\newcommand{\DD}{\displaystyle}
\newcommand{\vardel}{\partial}

\newcommand{\Moduli}{\overline{\mathcal{M}}}
\newcommand{\moduli}{\mathcal{M}}

\newcommand{\A}{\mathcal{A}}

\newcommand{\calP}{\mathcal{P}}
\newcommand{\calQ}{\mathcal{Q}}
\newcommand{\calF}{\mathcal{F}}
\newcommand{\calL}{\mathcal{L}}

\newcommand{\calA}{\mathcal{A}}
\newcommand{\calB}{\mathcal{B}}
\newcommand{\calH}{\mathcal{H}}
\newcommand{\calD}{\mathcal{D}}
\newcommand{\calC}{\mathcal{C}}
\newcommand{\calW}{\mathcal{W}}

\newcommand{\cycF}{\mathcal{F}_{cyc}}
\newcommand{\cycH}{\mathcal{H}_{cyc}}
\newcommand{\cycL}{\mathcal{L}_{cyc}}
\newcommand{\cycQ}{\mathcal{Q}_{cyc}}

\newcommand{\decoratenew}{\frac{\psi_{\bullet_j}^{\alpha_j}-(-\psi_{\star_j})^{\alpha_j}}{-\psi_{\bullet_j}-\psi_{\star_j}}}

\newcommand{\decorateonenew}{\frac{\psi_{\bullet_1}^{\alpha_1}-(-\psi_{\star_1})^{\alpha_1}}{-\psi_{\bullet_1}-\psi_{\star_1}}}

\title[]{Wall-crossings for Hassett descendant potentials}

\author{Vance Blankers}
\address{Department of Mathematics, Colorado State University, Fort Collins, Colorado 80523-1874}
\email{\href{mailto:blankers@mail.colostate.edu }{blankers@mail.colostate.edu }}
\thanks{V.B. was supported by NSF FRG grant 1159964 (PI: Renzo Cavalieri)}

\author{Renzo Cavalieri}
\address{Department of Mathematics, Colorado State University, Fort Collins, Colorado 80523-1874}
\email{\href{mailto:renzo@math.colostate.edu}{renzo@math.colostate.edu}}
\thanks{R.C. acknowledges support from Simons Foundation Collaboration Grant 420720}

\subjclass[2010]{14N10,14N35}

\begin{document}
\allowdisplaybreaks
\setcounter{tocdepth}{1}

\maketitle

\begin{abstract}
This paper solves the combinatorics relating the intersection theory of $\psi$-classes of Hassett spaces to that of $\Moduli_{g,n}$. A generating function for intersection numbers of $\psi$ classes  on all Hassett spaces is obtained  from the Gromov-Witten potential of a point via a non-invertible transformation of variables. When restricting   to diagonal weights, the changes of variables are invertible and explicitly described as polynomial functions. Finally,  the comparison of potentials  is extended to the level of cycles: the pinwheel cycle potential, a generating function for  tautological classes  of rational tail type on $\Moduli_{g,n}$ is the right instrument to describe the pull-back to $\Moduli_{g,n}$ of all monomials of $\psi$ classes on  Hassett spaces.
\end{abstract}

\tableofcontents

\section*{Introduction}\label{sec:intro}

\subsection*{Overview of Results}
In this section we give an overview of the results of this paper, and  a non-technical paraphrase of the statements  of the main theorems.

The goal of this paper is to solve the combinatorics relating intersections of $\psi$ classes on Hassett spaces to intersections of $\psi$ classes on the Deligne-Mumford spaces $\Moduli_{g,n}$. For any weight data $\calA$, there is a natural contraction morphism $c_{\calA}:\Moduli_{g,n}\to \Moduli_{g, \calA}$, and one can compare a $\psi$  
class on $\Moduli_{g,n}$ with the pull-back of the corresponding class on $\Moduli_{g,\calA}$. It is well-known (recalled in Lemma \ref{lem:pull-back}) that the difference may be expressed as a sum of irreducible boundary divisors of rational tail type.

One knows how to multiply $\psi$ classes and boundary strata as elements of a decorated graph algebra $\mathfrak{G}$, whose elements yield representatives for tautological classes in the Chow ring of $\Moduli_{g,n}$ (see \cite{gp:nontaut}); comparing an arbitrary monomial in $\psi$ classes with the pull-back of its counterpart in a Hassett space may be performed systematically as a  computation in $\mathfrak{G}$; however, the combinatorial complexity grows very fast with respect to the number of marked points and the degree of the monomial. Since the tautological ring  of $\Moduli_{g,n}$ is a quotient of $\mathfrak{G}$, we exploit geometric relations in $R^\ast(\Moduli_{0,n})$ to give a precise formula expressing this comparison.
We paraphrase our first main result.

\begin{thmf}
For any weight data $\calA$,  and any multi-index $I$, the difference $\psi^{I}- c_\calA^\ast(\psi_{\calA}^I)$ is represented by a sum of decorated boundary strata, with the following features:
\begin{itemize}
	\item the sum is supported on {\bf pinwheel strata} (see Def. \ref{defpin}), parameterizing stable curves with one component of genus $g$ and rational components attaching only to the component of genus $g$.
	\item the pinwheel strata contributing non-trivially  are indexed by {\bf $\A$-totally unstable partitions} (see Def. \ref{totun}).
	\item the dual graphs of the strata are appropriately decorated by polynomials in $\psi$ classes on edges and legs incident to the unique vertex of genus $g$. 
\end{itemize}
\end{thmf}

The formulas in Theorem \ref{thm:hassettcycle} highlight that the combinatorics comparing intersection of $\psi$ classes on Hassett and Deligne-Mumford spaces is independent of genus, number of marked points and the actual monomials that are being compared. This statement is made precise in the language of generating functions. For the sake of familiarity, we first restrict our attention to numerical intersections (cycles of degree equal to the dimension of the moduli space they are supported on).

\begin{lm2}
Denote by $\calH(\lambda;\mathbf{t})$ (the {\bf Hassett potential}, Def. \ref{def:hassettpotential}) a generating function for intersection numbers of $\psi$ classes on all Hassett spaces, and by $\calF(\lambda;\mathbf{x})$ (the {\bf Witten potential}, Def. \ref{def:gwpotential}) a generating function for intersection numbers of $\psi$ classes on $\Moduli_{g,n}$'s.
Then $\calH$ is obtained from $\calF$ via a non-invertible transformation of variables.
\end{lm2}

The change of variables statement follows naturally from  Theorem \ref{thm:wittentohassett}, which interprets the combinatorial formulas from  Theorem \ref{thm:hassettcycle} as relating the two potentials via the action of the exponential of a vector field. The shape of the vector field determines explicit formulas for the transformation of variables. While general formulas are cumbersome, the situation becomes elegant when restricting attention to Hassett spaces where all points have the same weight.

\begin{thm3}
Denote by $\calD_q(\lambda;\mathbf{t})$ (the {\bf $q$-diagonal Hassett potential}, Def. \ref{def:qhpot}) a generating function for intersection numbers of $\psi$ classes on Hassett spaces where all points have weight $\frac{1}{q}$. Then $\calD_q(\lambda;\mathbf{t})$ and $\calF(\lambda;\mathbf{x})$ are related by an invertible, polynomial change of variables, made explicit in Corollary \ref{cor:changeofvariablesdiag}.
\end{thm3}

 Hassett spaces with diagonal weights may be parameterized  by a rational, half line $\mathbb{Q} \cap [1, \infty)$,  with coordinate $a$  corresponding to the inverse of the weight of the points; for all natural numbers $q$, the Hassett moduli spaces are isomorphic in intervals $[q, q+1)$, and
 Theorem \ref{thm:wittentodiag} describes the wall-crossings  for descendant potentials in this family of moduli spaces. 

These results can be lifted to the level of cycles. After defining appropriate cycle-valued generating functions, one obtains that the combinatorics from Theorem \ref{thm:wittentohassett} is encoded in terms of the action of a differential operator on a potential of decorated pinwheel strata.

\begin{thm4}
Denote by $\cycH(\lambda;\mathbf{t})$ (the {\bf cycle Hassett potential}, Def. \ref{def:cychas}) a cycle valued generating function for intersections of $\psi$ classes on all Hassett spaces, and by $\cycQ(\lambda;\mathbf{x})$ (the {\bf cycle pinwheel potential}, Def. \ref{def:cycpin}) a generating function for intersections of $\psi$ classes on $\Moduli_{g,n}$'s. There exists a vector field $\cycL$ (the {\bf cycle fork operator}, Def. \ref{def:cfork}) such that the exponential of $\cycL$ applied to $\cycQ$ produces $\cycH$.
\end{thm4}

\subsection*{Context, Connections and Motivation}

Hassett spaces, or spaces of weighted marked stable curves, were first introduced in \cite{hassett} as a family of compactifications of $\moduli_{g,n}$ intended to study the minimal model program for moduli spaces of curves. Since then, there has been continued interest in understanding these spaces and their relationships to the Deligne-Mumford compactification $\Moduli_{g,n}$. In \cite{moon15}, Hassett spaces are shown to be log canonical models of $\Moduli_{g,n}$.

The tautological classes $\psi_i$  on $\Moduli_{g,n}$  (which we refer to as ordinary $\psi$ classes) are of fundamental importance in the study of the tautological ring of the moduli space of curves: they control the non-transversal intersections of boundary strata, and their push-forwards with respect to gluing and forgetful morphisms provide a natural set of additive generators for the tautological ring (\cite[Proposition 11]{gp:nontaut}). Further, intersection numbers of $\psi$ classes exhibit remarkable combinatorial and algebraic structure:  the content  of Witten's Conjecture/Kontsevich's Theorem (\cite{witten, kontsevich}) is that  a generating function for these intersection numbers is a $\tau$-function for the KdV hierarchy.

Hassett spaces define a family of birational models of $\Moduli_{g,n}$ paramerized by the chambers of a hyperplane arrangement in an $n$-dimensional rational orthant. Tautological cotangent line bundles for each mark, and therefore $\psi$ classes, are defined for all Hassett spaces (we refer to them as weighted $\psi$ classes for brevity). It is a very natural question to investigate how the intersection theory of $\psi$ classes  varies among Hassett spaces. Two $\psi$ classes on different Hassett spaces may be pulled-back to $\Moduli_{g,n}$, and their difference consists of a linear combination of boundary divisors of rational tail type. Comparing intersection cycles of $\psi$ classes on different Hassett spaces may be viewed as a combinatorial multiplication problem in a quotient of an algebra of decorated graphs of rational tail type. This problem was solved for intersection numbers (push-forward to a point  of zero dimensional cycles)  in \cite[Theorem 7.9]{alexeevguy}, where the authors give a precise formula relating any top intersection number of weighted $\psi$ classes in terms of intersection numbers of $\psi$ classes on $\Moduli_{g,n}$'s. 

We generalize the results of \cite{alexeevguy} in two ways: first, we organize  intersection numbers in generating functions and show that the combinatorial formulas relating intersection numbers  of weighted and ordinary $\psi$ classes are well-tuned with this organization. Second, we generalize the comparison from intersection numbers to cycles of arbitrary degree.

Generating functions are a powerful tool in enumerative geometry and combinatorics. Intersection numbers of weighted $\psi$ classes become coefficients of a formal power series $\calH$ in a countable number of variables indexed by a pair $(a,k)$, where $a$ is a rational number corresponding to the weight of a mark, and $k$ is a non-negative integer keeping track of the power of the $\psi$ class at that mark. Similarly, intersection numbers for ordinary $\psi$ classes are encoded in a potential $\calF$ where the variables only need the index $k$ (See Definitions \ref{def:hassettpotential} and \ref{def:gwpotential}). All combinatorial formulas relating intersection numbers of weighted and ordinary $\psi$ classes are encoded in the statement that the function $\calH$ is obtained from $\calF$ by a transformation of variables. Further, we provide a conceptually satisfying proof of this statement: the combinatorial formulas can naturally be described as the action of the exponential of a vector field $\calL$ on  $\calH$, and the change of variable is a direct consequence of exponential flow along  $\calL$.
A natural way to avoid dealing with a doubly-countable set of variables is to restrict attention to a one dimensional path in the parameter space for Hassett spaces. In this paper we choose to analyze the small diagonal line, parameterizing moduli spaces where all points have equal weights. We obtain a wall-crossing picture where the generating functions $\calD_q$ in each chamber are related to the ordinary Gromov-Witten potential $\calF$  by invertible polynomials, which are successive truncations of power series that can be described in terms of Schur polynomials. This is related to the wall-crossings of quasimaps in \cite{cfk:qm1,cfk:qm2}: on the one hand our picture is more restrictive since  Hassett spaces of diagonal weights may be compared to quasimaps to  a point target; on the other,  our statement extends the wall-crossing picture from the $J$ function to the full descendant potential.

In recent years, the scope of Gromov-Witten theory and related areas has vastly enlarged, both in terms of studying many variants of moduli spaces of curves and maps, and in terms of going beyond the enumerative geometry of zero-dimensional intersection cycles to study families of tautological cycles of arbitary degree, as surveyed in \cite{pandhamg}.  A very powerful technique to deal with the combinatorial complexity of families of tautological classes is the action of an $R$-matrix on cohomological field theories (CohFT), \cite{ppz:spin}. Under suitable circumstances this  yields graph formulas (\cite{CavTarasca}) describing such families. The collection of all intersection cycles of $\psi$ classes may be organized to satisfy most, but not all axioms of a CohFT: depending on the weight of the points, either the pull-back axiom fails, or the splitting axiom fails for divisors of rational tails. Nonetheless, our comparison of  intersection cycles of $\psi$ classes is given as a {\bf graph formula}, where the edge decorations are very much suggestive of the (logarithm of the) action of an $R$-matrix, or, equivalently, to the action of the quantized operator $\hat{z}^k$ (see \cite{giv:qqh}). It would be very interesting, but is currently beyond our reach,  to make such  observations precise and see what consequences they entail about the structure of  cycle potentials.

This works completes a program that had begun with the two most special cases:  \cite{nand} studies Hassett spaces with all weights equal to $1/2$, where the boundary corrections consist only of attaching tripods; \cite{bckappa} explores the asymptotic limit of Hassett spaces as the weights become infinitesimal, and  connects it to the intersection theory of $\kappa$ classes on $\Moduli_g$.

We expect that these techniques may be extended to further families of birational compactifications of moduli spaces, where the stability conditions are modified in terms of some well-structured, combinatorial data. For example, \cite{fry} is beginning the exploration of moduli spaces of pointed curves where the stability is encoded in a graphic matroid.

\subsection*{Acknowledgments} The authors wish to thank Nick Ercolani, Andreas Gross, Y.P. Lee and Hannah Markwig for many helpful discussions related to this project.

\section{Preliminaries on $\Moduli_{g,n}$}\label{sec:prelimmgn}

We assume some familiarity with $\moduli_{g,n}$ and its Deligne-Mumford compactification $\Moduli_{g,n}$; otherwise, we recommend \cite{harrismorrison,vakil08} as excellent resources. This section serves  to highlight some of the facts particularly relevant to our purposes and to establish notation.

The moduli space of isomorphism classes of stable genus $g$ curves with $n$ (ordered, distinct, smooth) marked points is denoted by $\Moduli_{g,n}$; if $2g-2+n > 0$, this space is a smooth Deligne-Mumford stack of complex dimension $3g-3+n$. We denote points of $\Moduli_{g,n}$ by $[C; p_1,\dots,p_n]$. For fixed $g$, we may vary $n$ to obtain a family of moduli spaces related by \emphbf{forgetful morphisms}: for each $1\leq i \leq n$, the map
\begin{align}
\pi_{i}:\Moduli_{g,n}\to\Moduli_{g,n-1}
\end{align}
forgets the $i$-th marked point and stabilizes the curve if necessary by contracting unstable rational components. The morphism $\pi_{i}$ functions as a universal family for $\Moduli_{g,n}$ by identifying the universal curve $\overline{\mathcal{U}}_{g,n} \to \Moduli_{g,n}$ with $:\pi_{n+1}:\Moduli_{g,n+1}\to \Moduli_{g,n}$. The \emphbf{$i$-th tautological section}
\begin{align}
\sigma_{i}:\Moduli_{g,n}\to \overline{\mathcal{U}}_{g,n} \cong \Moduli_{g,n+1}
\end{align}
assigns to an $n$-marked curve $[C;p_1,\dots,p_n]$ the point $p_i$ in the fiber over $[C;p_1,\dots,p_n]$ in the universal curve.

There  are \emphbf{gluing morphisms}
\begin{align}
gl:\Moduli_{g_1,n_1+1} \times \Moduli_{g_2,n_2+1} \to \Moduli_{g_1+g_2,n_1+n_2}
\end{align}
and \emphbf{clutching morphisms}
\begin{align}
cl: \Moduli_{g-1,n+2} \to \Moduli_{g,n}.
\end{align}
The push-forward of the fundamental class under a gluing or clutching morphism is (a multiple of) an irreducible boundary divisor. Together, the forgetful, gluing, and clutching morphisms are the \emphbf{tautological morphisms}.

Due to the significant complexity of the full Chow ring of $\Moduli_{g,n}$, the \emphbf{tautological subring} $R^*(\Moduli_{g,n})$ is often considered instead (for both rings we assume rational coefficients). This ring is simultaneously defined for all $g$ and $n$ such that $2g-2+n>0$ by setting $R^*(\Moduli_{g,n})$ to be the smallest subring of $A^*(\Moduli_{g,n})$ containing the fundamental class $[\Moduli_{g,n}]$ which is closed under push-forward by the tautological morphisms. Of particular interest in this paper are a family of tautological classes called $\psi$-classes.

\begin{definition}
For $i\in [n] = \{1,2,\dots, n\}$, the class $\psi_i\in R^1(\Moduli_{g,n})$ is defined to be
\begin{align}
\psi_i:= c_1(\sigma_i^*(\omega_\pi)),
\end{align}
where $\omega_\pi$ denotes the relative dualizing sheaf of the universal family $\pi:\overline{\mathcal{U}}_{g,n} \to \Moduli_{g,n}$ and $c_1$ is the first Chern class.
\end{definition}

Certain boundary divisors on $\Moduli_{g,n}$ are given special notation. For  $P\sqcup Q = [n]$ and $|Q|\geq 2$, $D(P\,|\,Q)$ denotes the class of the boundary divisor which generically parametrizes nodal curves with a genus $g$ component containing the marks specified by $P$ attached to a rational tail containing the marks specified by $Q$. Such a divisor class $D = D(P\,|\,Q)$ is the pushfoward along the gluing morphism $gl_D:\Moduli_{g,P\cup\{\bullet\}}\times \Moduli_{0,Q\cup\{\star\}} \to \Moduli_{g,n}$ of the fundamental class.

The $\psi$-classes are not stable under pull-back by forgetful morphisms, as described in the following lemma.

\begin{lemma}[\cite{kockpsi}] \label{lem:psicomp} 
Consider the forgetful morphism $\pi_{n+1}: \Moduli_{g,n+1} \to \Moduli_{g,n}$. For $i\in [n]$, we have:

\hfill
$\psi_i = \pi_{n+1}^\ast (\psi_i) + D\bigg([n]\smallsetminus\{i,n+1\}\,\bigg|\,\{i,n+1\}\bigg).$
\qed

\end{lemma}

This gives a canonical boundary expression for the sum of two $\psi$-classes on $\Moduli_{0,n}$ as an easy corollary.

\begin{lemma}[{{\cite[Lemma 1.4]{bcomega}}}] \label{lem:psisums}
Let $n\geq 3$; for any distinct $i,j\in [n]$ the following idenitity holds in the tautological ring of  $\Moduli_{0,n}$:

\hfill
$\DD \psi_i + \psi_j = \sum_{\tiny{\begin{array}{c} P\ni j \\ Q \ni  i \end{array}}} D(P\,|\,Q).$
\qed
\end{lemma}

\section{Hassett Spaces}\label{sec:hassettbackground}

This section contains basic definitions and relevant background on Hassett spaces.

\begin{definition}
Fix \emphbf{(ordered) weight data} $\calA = (a_1,a_2,\dots,a_n)$ so that $a_i\in (0,1]\cap\mathbb{Q}$ and $a_i \geq a_{i+1}$. A (nodal) marked curve $(C;p_1,\dots,p_n)$ is \emphbf{$\A$-stable} if 
\begin{itemize}
\item $p_i\in C$ is smooth for every $i\in [n]$;
\item $\DD \omega_C + \sum_{i=1}^n a_ip_n$ is ample; and
\item for every point $x\in C$, we have $\DD\sum_{p_i=x}a_i \leq 1$.
\end{itemize}
The  \emphbf{Hassett space} $\Moduli_{g,\calA}$ is the moduli space of $\A$-stable curves of genus $g$ up to isomorphism. When weight data is \emphbf{diagonal}, i.e. $a_1 = a_2 = \cdots a_n = a$, we write $\calA = a^n$.
\end{definition}

\begin{remark}
Define the \emphbf{total weight} of a point $x\in C$ to be $\DD\sum_{p_i=x} a_i$ if $x$ is smooth and $1$ if $x$ is a node. The total weight of an irreducible component is the sum of the total weights of all of its points. The second condition of $\calA$-stability may be reinterpreted to say that the rational components of an $\A$-stable curve $C$ have total weight greater than two and the genus one components have total weight greater than zero.
\end{remark}

When $2g-2+\sum a_i > 0$, the Hassett space $\Moduli_{g,\calA}$ is a non-empty, smooth, proper Deligne-Mumford stack. In analogy with $\Moduli_{g,n}$,  Hassett spaces have forgetful morphisms; for generic weight data \footnote{generic  means that there is no subset $J \subseteq [n]$ with $\sum_{j\in J} a_j = 1$. One may remove the genericity requirement by allowing points to have weight $0$; see \cite{hassett}. }
\begin{align}
\pi_{n+1}:\overline{\mathcal{U}}_{g,\calA} \to \Moduli_{g,\calA}
\end{align}
identifies the universal curve with $\Moduli_{g,\hat{\calA}}$ for $\hat{\calA} = (a_1,\dots,a_n,\epsilon)$ for $\epsilon$ sufficiently small. The universal family has $n$ \emphbf{tautological sections}
\begin{align}
\sigma_i:\Moduli_{g,\calA} \to \overline{\mathcal{U}}_{g,\calA},
\end{align}
where the $i$-th section assigns to $[C; p_1,\dots,p_n]$ the point $p_i$ in the fiber over $[C;p_1,\dots,p_n]$ in the universal curve.

Hassett spaces also have gluing and clutching morphisms and a \emphbf{tautological subring} $R^*(\Moduli_{g,\calA}) \subset A^*(\Moduli_{g,\calA})$, which is defined analogously to the tautological ring for $\Moduli_{g,n}$.

The Hassett space $\Moduli_{g,1^n}$ coincides with the Deligne-Mumford compactification $\Moduli_{g,n}$. More generally, the key difference between $\calA$-stable curves and Deligne-Mumford stable curves is that the former may have marked points come together provided the combined weight is low enough.

For fixed $n$, there is a partial ordering on weight data: $\calA \succeq \calB$ if $a_i\geq b_i$ for all $i\in [n]$. Given $\calA\succeq \calB$, there exists a birational \emphbf{reduction morphism}
\begin{align}
r_{\calB,\calA}:\Moduli_{g,\calA}\to\Moduli_{g,\calB},
\end{align}
which on the level of curves reduces the weights and stabilizes if necessary by contracting unstable rational components (see Figure \ref{fig:reduction}). The space $\Moduli_{g,n}$ admits reduction morphisms to all other Hassett spaces; for a given $\calA$ we call  this reduction a \emphbf{contraction morphism} and denote it
\begin{align}
c_{\calA}:\Moduli_{g,n}\to\Moduli_{g,\calA}.
\end{align}

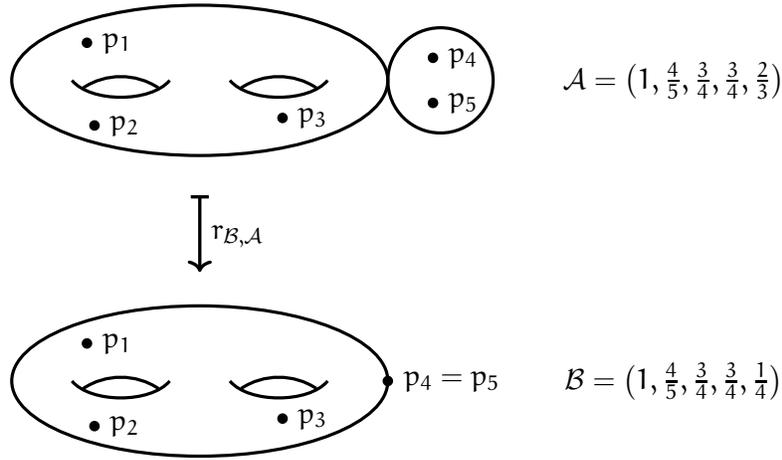
\begin{figure}[t]
\begin{tikzpicture}


\draw[very thick] (0,0) ellipse (2.5cm and 1cm);

\draw[very thick] (3.2,0) circle (0.7);

\draw[very thick] (-1.55,-0.1) .. controls (-1.25,0.1) and (-0.85,0.1) .. (-0.55,-0.1); 
\draw[very thick] (-1.7,0) .. controls (-1.45,-0.3) and (-0.65,-0.3) .. (-0.4,0); 

\draw[very thick] (1.55,-0.1) .. controls (1.25,0.1) and (0.85,0.1) .. (0.55,-0.1); 
\draw[very thick] (1.7,0) .. controls (1.45,-0.3) and (0.65,-0.3) .. (0.4,0); 

\fill (-1.5,0.5) circle (0.07); \node at (-1.1,0.5) {$p_1$};
\fill (-1.4,-0.6) circle (0.07); \node at (-1,-0.6) {$p_2$};
\fill (1.1,-0.5) circle (0.07); \node at (1.5,-0.5) {$p_3$};
\fill (3.1,0.3) circle (0.07); \node at (3.5,0.3) {$p_4$};
\fill (3.1,-0.3) circle (0.07); \node at (3.5,-0.3) {$p_5$};

\node at (6.3,0) {$\calA = \left(1,\frac{4}{5},\frac{3}{4}, \frac{3}{4}, \frac{2}{3}\right)$};


\draw[->,line width=1.25pt] (0,-1.55) -- (0,-2.55); \node at (0.5,-2.05) {$r_{\calB,\calA}$};
\draw[line width=1.25pt] (-0.12,-1.55) -- (0.12,-1.55);


\draw[very thick] (0,-4) ellipse (2.5cm and 1cm);

\draw[very thick] (-1.55,-4.1) .. controls (-1.25,-3.9) and (-0.85,-3.9) .. (-0.55,-4.1); 
\draw[very thick] (-1.7,-4) .. controls (-1.45,-4.3) and (-0.65,-4.3) .. (-0.4,-4); 

\draw[very thick] (1.55,-4.1) .. controls (1.25,-3.9) and (0.85,-3.9) .. (0.55,-4.1); 
\draw[very thick] (1.7,-4) .. controls (1.45,-4.3) and (0.65,-4.3) .. (0.4,-4); 

\fill (-1.5,-3.5) circle (0.07); \node at (-1.1,-3.5) {$p_1$};
\fill (-1.4,-4.6) circle (0.07); \node at (-1,-4.6) {$p_2$};
\fill (1.1,-4.5) circle (0.07); \node at (1.5,-4.5) {$p_3$};
\fill (2.5,-4) circle (0.07); \node at (3.35,-4) {$p_4=p_5$};

\node at (6.3,-4) {$\calB = \left(1,\frac{4}{5},\frac{3}{4}, \frac{3}{4}, \frac{1}{4}\right)$};

\end{tikzpicture}
\caption{An example of the action of a reduction morphism on curves. On top, $a_4+a_5 > 1$, but on the bottom, $b_4+b_5 \leq 1$.}
\label{fig:reduction}
\end{figure}

All appropriate forgetful, reduction, contraction morphisms commute and induce maps on the tautological rings of the relevant Hassett spaces. The contraction morphisms define special boundary strata in $\Moduli_{g,n}$.

\begin{definition} \label{def:unstablestrat}
A boundary stratum $\Delta$ in $\Moduli_{g,n}$ is \emphbf{$\calA$-stable} if the topological type of the general curve parametrized by $\Delta$ does not change under the application of $c_{\calA}$. Otherwise $\Delta$ is \emphbf{$\calA$-unstable}.
\end{definition}

Hassett spaces also admit $\psi$-classes, which are the protagonists of the remainder of the paper.

\begin{definition}
Fix weight data $\calA$. For each $i\in [n]$, the class $\psi_{i,\calA}\in R^1(\Moduli_{g,\calA})$ is defined to be
\begin{align*}
\psi_{i,\calA} := c_1(\sigma_i^*(\omega_{\pi}))
\end{align*}
where  and $c_1$ is the first Chern class, $\omega_\pi$ denotes the relative dualizing sheaf of the universal family $\pi:\overline{\mathcal{U}}_{g,\calA} \to \Moduli_{g,\calA}$, and $\sigma_i:\Moduli_{g,\calA}\to\overline{\mathcal{U}}_{g,\calA}$ is the $i$-th tautological section. These are called \emphbf{$\calA$-weighted $\psi$-classes} or just \emphbf{weighted $\psi$-classes}.
\end{definition}

\begin{remark}
\label{rem:omega}
When $\calA = 1^n$, we omit the weight data and write $\psi_{i,1^n} = \psi_i$, since this is the standard $i$-th $\psi$-class on $\Moduli_{g,n}$.
\end{remark}

Weighted $\psi$-classes are generally unstable under pull-back along reduction morphisms in a way reminiscent of unweighted $\psi$-classes and forgetful morphisms. The next lemma makes this precise.

\begin{lemma}[{{\cite[Lemma 5.3]{alexeevguy}}}] \label{lem:pull-back}
For weight data $\calA \succeq \calB$,
\begin{align}
\psi_{i,\calA} = r_{\calB,\calA}^*\psi_{i,\calB} + D,
\end{align}
where $D$ is the sum of all boundary divisors whose generic element is a nodal curve with a genus $g$ component attached to a rational tail containing the $i$-th marked point and which is $\calA$-stable but not $\calB$-stable.
\hfill $\square$
\end{lemma}

When $r_{\calB,\calA} = c_{\calB}$, the condition on the divisorial correction simplifies to be the sum of divisors  which are $\calB$-unstable in the sense of Definition \ref{def:unstablestrat}.

\begin{example}
\label{ex:wtpsi}
Let $\calB = \left(\frac{7}{8},\frac{2}{3},\frac{1}{3},\frac{1}{4}, \frac{1}{6}\right)$.
\begin{align*}
c_{\calB}^*\psi_{1,\calB} &= \psi_{1} \\\
c_{\calB}^*\psi_{2,\calB} &= \psi_{2} - D(\{1,4,5\}\,|\,\{2,3\}) - D(\{1,3,5\} \,|\, \{2,4\}) - D(\{1,3,4\}\,|\,\{2,5\}) \\
c_{\calB}^*\psi_{3,\calB} &= \psi_{3} - D(\{1,4,5\}\,|\,\{2,3\}) - D(\{1,2,5\}\,|\,\{3,4\}) - D(\{1,2,4\}\,|\,\{3,5\}) \\
& \hspace{7.5cm} - D(\{1,2\}\,|\,\{3,4,5\})
\end{align*}
$\psi_{4,\calB}$ and $\psi_{5,\calB}$ pull back analogously to $\psi_{3,\calB}$.
\end{example}

The next two lemmas describe how weighted $\psi$-classes restrict to boundary divisors depending upon the $\calA$-stability of the divisors; the former concerns $\calA$-stable divisors and the latter $\calA$-unstable divisors.



\begin{lemma} \label{lem:donothing}
Fix weight data $\calA = (a_1,\dots,a_n)$. Let $D = D(P | Q)$ be an $\calA$-stable divisor of rational tails type on $\Moduli_{g,n}$. Then for any non-negative integer $k$
\begin{align}
c^*_\calA\psi_{i,\calA}^k\cdot D = {gl_{D}}_\ast (\psi^k_{i}).
\end{align}
\end{lemma}

\begin{proof}
Let $\calA_P$ (resp. $\calA_{Q}$) denote the weight data corresponding to points in $P$ (resp. $Q$), and let $D_\calA$ be the push-forward of $D$ under $c_{\calA}$. Because $D$ is $\calA$-stable, the restriction $c_{\calA}\big\vert_{D}$ is an isomorphism onto its image, and the claim follows from the commutivity of the following diagram (where we assume $\bullet_\calA$ and $\star_\calA$ have weight 1).
\begin{center}
\begin{tikzcd}[row sep=2cm, column sep=1.5cm]
\Moduli_{g,P\cup\{\bullet\}}\times \Moduli_{0,Q\cup\{\star\}} \arrow{d}{gl_D} \arrow[r, "c_\calA\vert_D"] & \Moduli_{g,\calA_P\cup\{\bullet_\calA\}}\times \Moduli_{0,\calA_Q\cup\{\star_\calA\}} \arrow[d, "gl_{D_\calA}"] \\
\Moduli_{g,n} \arrow[r, "c_\calA"] & \Moduli_{g,\calA}
\end{tikzcd}
\end{center}
\end{proof}

\begin{lemma} \label{lem:falldownhassett}
Fix weight data $\calA = (a_1,\dots,a_n)$. Let $D = D(P|Q)$ be an $\calA$-unstable divisor of rational tails type on $\Moduli_{g,n}$, and suppose the $i$-th marked point is on the rational component ($i\in Q$). Then for any non-negative integer $k$ 
\begin{align}
c^*_{\calA}\psi_{i,\calA}^k\cdot D = {gl_{D}}_\ast (c^*_{\mathcal{B}} \psi^k_{\bullet,\mathcal{B}}),
\end{align}
where $c^*_{\mathcal{B}}\psi_{\bullet,\mathcal{B}}$ denotes the class pulled back from the projection $pr: \Moduli_{g, P \cup \{\bullet\}} \times \Moduli_{0, Q \cup \{\star\}} \to \Moduli_{g,P \cup \{\bullet\}}$, and $\mathcal{B} = (a_{i_1},\dots,a_{\ell}, s)$ with $s = \min\left(1,\sum_{i\in Q} a_i\right)$.
\end{lemma}

\begin{proof}
Without loss of generality, let $i=n$. Fix a divisor $D(P|Q)$ of rational tails type on $\Moduli_{g,n}$ with $P = \{i_1,\dots,i_\ell\}$ and $Q = \{i_{\ell+1},\dots,i_{n}=n\}$, and suppose that $D(P|Q)$ is $\calA$-unstable. Let $s = \min\left(1,\sum_{i\in Q} a_i\right)$, define $\widetilde{D} := {c_{\calA}}(D) = \{[C;p_1,\dots,p_n] : p_{i_{\ell+1}} = \cdots = p_{i_n}\} \subset \Moduli_{g,\calA}$, and define weight data $\mathcal{B} = (a_{i_1},\dots,a_{\ell}, s)$.

Consider the following diagram ($\iota$ is inclusion):
\begin{center}
\begin{tikzcd}[row sep=1cm, column sep=2cm]
\Moduli_{g,P\cup\{\bullet\}}\times \Moduli_{0,Q\cup\{\star\}} \arrow{dd}{gl_D} \arrow[r, "pr"] & \Moduli_{g,P\cup\{\bullet\}} \arrow[d, "c_{\calB}"] \\
& \Moduli_{g,\mathcal{B}}\cong \widetilde{D} \arrow[d, "\iota"] \\
\Moduli_{g,n} \arrow[r, "c_\calA"] & \Moduli_{g,\calA}
\end{tikzcd}
\end{center}

By the commutativity of the diagram,
\begin{align*}
c^*_{\calA}\psi_{n,\calA}^k \cdot D(P|Q) &= {gl_{D}}_\ast gl_D^*(c^*_{\calA}\psi^k_{n,\calA} )\\
&= {gl_{D}}_\ast pr^*c^*_{\calB}\iota^*(\psi^k_{n,\calA}) \\
&= {gl_{D}}_\ast(c^*_{\calB}\psi^k_{\bullet,\calB}).
\end{align*}
\end{proof}


\begin{remark}
In order to streamline notation, when we write $\psi$-classes relative to some flag of a dual graph, we implicitly mean the push-forward via the appropriate gluing morphism of the pull-back via the projection to the factor hosting the flag of the corresponding class. See Figure \ref{fig:not} for an illustration.
\end{remark}

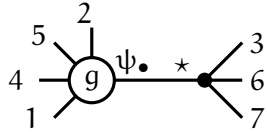
\begin{figure}[h]

\begin{tikzpicture}
\draw[very thick]  (0,0) circle (0.30);
\node at (0,0) {$g$};
\draw[very thick] (0.30,0) -- (1.5,0);
\fill (1.5,0) circle (0.10);
\draw[very thick] (1.5,0) -- (2,0);
\draw[very thick] (1.5,0) -- (2,0.5);
\draw[very thick] (1.5,0) -- (2,-0.5);
\node at (2.2,0) {$6$};
\node at (2.2,0.5) {$3$};
\node at (2.2,-0.5) {$7$};
\node at (0.55,0.25) {$\psi_\bullet$};
\node at (1.2,0.2) {$\star$};

\draw[very thick] (0,0.30) -- (0,0.7);
\node at (-0.1,0.9) {$2$};

\draw[very thick] (-0.21,0.21) -- (-0.50,0.5);
\node at (-0.7,0.6) {$5$};

\draw[very thick] (-0.30,0) -- (-.7,0);
\node at (-1,0) {$4$};

\draw[very thick] (-0.21,-0.21) -- (-.5,-.5);
\node at (-0.8,-.5) {$1$};

\end{tikzpicture}
\caption{The dual graph identifying the divisor $D = D(P\,|\,Q)$, with $P= \{1,2,4,5\}$ and $Q= \{ 3,6,7\}$. The graph is decorated with a $\psi$ class on a flag. This is shorthand for ${gl_{D}}_\ast p_1^\ast(\psi_\bullet)$.
}\label{fig:not}
\end{figure}

\section{Cycle Formula for Intersections}\label{sec:cycle}

Intersections of pull-backs of $\calA$-weighted $\psi$-classes have a highly combinatorial structure which relates them to classical $\psi$-class intersections. 
We define the strata in $R^*(\Moduli_{g,n})$ which support such intersections; throughout, we say a partition $\calP = \{P_1,\dots,P_r\}$ has $r$ \emphbf{parts} and has \emphbf{length} $\ell(\calP) = r$.


\begin{definition}\label{defpin}
Given a partition $\mathcal{P} = \{P_1,\dots,P_r\} \vdash [n]$, when $|P_i| = 1$  denote by $\bullet_i$ the element of the singleton $P_i$.  For $|P_i|>1$, introduce new labels $\bullet_i$ and $\star_i$. The \emphbf{pinwheel stratum} $\Delta_{\mathcal{P}}$ is the image of the gluing morphism
\begin{align}
gl_{\mathcal{P}}: \Moduli_{g, \{\bullet_1, \ldots, \bullet_r\}} \times \prod_{|P_i|>1} \Moduli_{0, \{\star_i\}\cup P_i} \to \Moduli_{g,n}
\end{align}
that glues together each $\bullet_i$ with $\star_i$.
Since  the general point of a pinwheel stratum represents a curve with no nontrivial automorphisms, the class of the stratum equals the push-forward of the fundamental class via $gl_{\mathcal{P}}$:
\begin{align}
\label{pinwheel}
[\Delta_{\mathcal{P}}] = gl_{\mathcal{P} \ast}([1]).
\end{align}
\end{definition}

\begin{example}
\label{ex:pinwheel}
Figure \ref{fig:pinwheel} shows  an example of the dual graph of a generic element of a pinwheel stratum. 

\begin{figure}[tb]

\begin{tikzpicture}

\draw[very thick]  (0,0) circle (0.30);
\node at (0,0) {$g$};

\draw[very thick] (0.30,0) -- (1.5,0);
\fill (1.5,0) circle (0.10);
\draw[very thick] (1.5,0) -- (2,0);
\draw[very thick] (1.5,0) -- (2,0.5);
\draw[very thick] (1.5,0) -- (2,-0.5);
\node at (2.2,0) {$6$};
\node at (2.2,0.5) {$3$};
\node at (2.2,-0.5) {$7$};
\node at (0.6,0.2) {$\bullet_4$};
\node at (1.2,0.2) {$\star_4$};

\draw[very thick] (0,0.30) -- (0,1.5);
\fill (0,1.5) circle (0.10);
\draw[very thick] (0,1.5) -- (-0.5,2);
\draw[very thick] (0,1.5) -- (0.5,2);
\node at (-0.5,2.2) {$2$};
\node at (0.5,2.2) {$5$};
\node at (-0.2, .6) {$\bullet_3$};
\node at (-0.2, 1.2) {$\star_3$};

\draw[very thick] (-0.30,0) -- (-.7,0);
\node at (-1.3,0) {$\bullet_2 = 4$};

\draw[very thick] (-0.21,-0.21) -- (-.5,-.5);
\node at (-1.1,-.5) {$\bullet_1 = 1$};

\end{tikzpicture}
\caption{The dual graph to the generic curve parameterized by the pinwheel stratum $\Delta_\mathcal{P}$, with $\mathcal{P} = \{\{1\},\{4\},\{2,5\},\{3,6,7\}\}$. The edges of the graph are decorated with auxiliary markings coming from the gluing morphism.}\label{fig:pinwheel}
\end{figure}
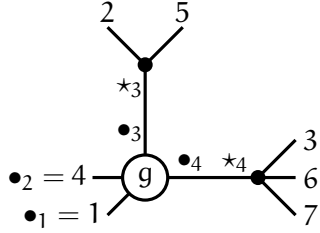
\end{example}


\begin{definition}\label{totun}
Fix weight data $\A = (a_1,\dots,a_n)$. A partition $\mathcal{P} = \{P_1,\dots,P_r\} \vdash [n]$ is \emphbf{$\A$-totally unstable} if $\DD\sum_{i\in P_j} a_i \leq 1$ for each part $P_j$. The set of all $\A$-totally unstable partitions is denoted $\mathfrak{P}_{\A}$.
\end{definition}



\begin{remark}
Equivalently, a partition $\mathcal{P}$ is $\calA$-totally unstable if the image of the corresponding pinwheel stratum via the contraction morphism lies in the interior of the Hassett space with weight data $\calA$: $$c(\Delta_\mathcal{P})\subset \mathcal{M}_{g, \calA}.$$
\end{remark}

If $\calA = 1^n$, then the only $\calA$-totally unstable partition is the singleton partition $\{\{1\},\dots,\{n\}\}$, and if $a_i\leq \frac{1}{n}$ for all $i$, then all partitions are $\calA$-totally unstable. The partial ordering on weight data is reversed for  the sets of unstable partitions: if $\calA \succeq \calB$, then $\mathfrak{P}_{\calA}\subseteq \mathfrak{P}_{\calB}$.

The intersection of weighted $\psi$-classes and pinwheel strata is analogous to that of weighted $\psi$-classes and divisors.

\begin{corollary} \label{cor:pinfall}
Fix weight data $\A = (a_1,\dots,a_n)$. For a pinwheel stratum $\Delta_\mathcal{P}$, if the $i$-th marked point is on a rational component which contracts under $c_{\calA}$, and without loss of generality $i\in P_1\in\calP$, then
\begin{align}
c^*_{\calA}\psi_{i,\calA}^k\cdot \Delta_\calP = {gl_{\calP}}_\ast (\psi^k_{\bullet_1}),
\end{align}
where $\psi_{\bullet_1}$ denotes the class pulled back from the projection 
\begin{align*}
pr:  \Moduli_{g, \{\bullet_1, \ldots, \bullet_r\}} \times \prod_{|P_i|>1} \Moduli_{0, \{\star_i\}\cup P_i} \to \Moduli_{g,\{\bullet_1,\dots,\bullet_r\}}.
\end{align*} Otherwise, if the $i$-th marked point is on a component which does not contract under $c_{\calA}$, then
\begin{align}
c^*_\calA\psi_{i,\calA}^k\cdot \Delta_\mathcal{P} = {gl_{\calP}}_\ast (\psi^k_{i}).
\end{align}
\end{corollary}

\begin{proof}
Both statements follow  from Lemma \ref{lem:donothing} and Lemma \ref{lem:falldownhassett}.
\end{proof}

We are now ready to state our first main theorem, which generalizes the result in \cite{bcomega} to arbitrary weight data.


\begin{theorem}
\label{thm:hassettcycle}
Fix $g$ and $n$ and weight data $\calA = (a_1,\dots,a_n)$. For $1\leq i \leq n$, let $k_i$ be a non-negative integer, and let $\DD K = \sum_{i=1}^n k_i$. For a partition $\mathcal{P}  = \{P_1,\dots, P_r \} \vdash [n]$, define $\DD \alpha_j := \sum_{i\in P_j} k_i$. Then the following formula holds in  $R^{K}(\Moduli_{g,n})$:
\begin{align}
\label{eq:hassettfor}
c^*_{\A}\left(\prod_{i=1}^n \psi_{i,\calA}^{k_i}\right) = \sum_{\mathcal{P}\in \mathfrak{P}_{\A}}[\Delta_{\mathcal{P}}] \prod_{|P_j| = 1}\psi_{\bullet_j}^{\alpha_j} \prod_{|P_j| > 1} \frac{\psi_{\bullet_j}^{\alpha_j}- (-\psi_{\star_j})^{\alpha_j}}{-\psi_{\bullet_j}-\psi_{\star_j}}.
\end{align}
\end{theorem}


\begin{remark} \label{rem:ratfun}
In formula \eqref{eq:hassettfor},  the rational function is really a polynomial:
\begin{equation}
\label{eq:ratfun1}
\frac{\psi_{\bullet_j}^{\alpha_j}- (-\psi_{\star_j})^{\alpha_j}}{-\psi_{\bullet_j}-\psi_{\star_j}} = - \psi_{\bullet_j}^{\alpha_j-1}+ \psi_{\bullet_j}^{\alpha_j-2} \psi_{\star_j} - \psi_{\bullet_j}^{\alpha_j-3} \psi_{\star_j}^{2}+ \dots - (-\psi_{\star_j})^{\alpha_j-1}.
\end{equation}
We also observe that if $\alpha_j=0$, the expression in \eqref{eq:ratfun1}  equals $0$. Hence  the formula is supported on pinwheel strata where each rational tail has at least one point $i$ with strictly positive $k_i$.
\end{remark}

\begin{proof}
The proof of Theorem \ref{thm:hassettcycle} proceeds by induction and entails plenty of careful bookkeeping. We have placed in Appendix \ref{hacyth} for the reader interested in all details.
\end{proof}

By restricting our attention to top-dimensional intersections, we recover as a corollary to Theorem \ref{thm:hassettcycle} the following numerical result originally shown in \cite{alexeevguy}.


\begin{corollary}
\label{cor:numerical}
For $1\leq i \leq n$, let $k_i$ be a non-negative integer, and let $\DD \sum_{i=1}^n k_i = 3g-3+n$. Then
\begin{align} \label{eq:numerical}
\int_{\Moduli_{g,\A}}\prod_{i=1}^n \psi_{i,\calA}^{k_i} = \sum_{\mathcal{P}\in\mathfrak{P}_{\A}}(-1)^{n+\ell(\mathcal{P})} \int_{\Moduli_{g,\ell(\mathcal{P})}}\prod_{i=1}^{\ell(\mathcal{P})} \psi_{i}^{\alpha_i-|P_i|+1}
\end{align}
\end{corollary}

\begin{proof}
This statement follows from formula \eqref{eq:hassettfor} by observing which monomials are of top degree and then pushing forward from $\Moduli_{g,n}$ to $\Moduli_{g,\A}$.
\begin{itemize}
\item For any partition $\mathcal{P}$, for dimension reasons the only monomial that has nonzero evaluation on $[\Delta_{\mathcal{P}}]$ is
\begin{align}\label{ha}
\prod_{|P_i|=1} \psi_{\bullet_i}^{\alpha_i}\prod_{|P_i|>1}(-1)^{|P_i|-1}\psi_{\bullet_i}^{\alpha_i-|P_i|+1}\psi_{\star_i}^{|P_i|-2 }
\end{align}
\item For any $n\geq 3$ and $i\in [n]$, 
\begin{align*}
\int_{\Moduli_{0,n}} \psi_i^{n-3}  = 1
\end{align*}
by the string equation; hence all evaluations for the classes $\psi_{\star_i}$ in \eqref{ha} contribute a factor of one to the evaluation of the monomial on $[\Delta_{\mathcal{P}}]$.
\end{itemize}
\end{proof}

\begin{remark}\label{rem:generalcycle}
It is worth observing that natural generalizations of Theorem \ref{thm:hassettcycle} and Corollary \ref{cor:numerical} hold as well: if instead of pulling back along $c_{\calA}$ to $\Moduli_{g,n}$ we pull back to some intermediate $\Moduli_{g,\calB}$ along $r_{\calA,\calB}$, the only change is that the sums in \eqref{eq:hassettfor} and \eqref{eq:numerical} are over 
\begin{align*}
\left(\mathfrak{P}_{\calA} \backslash \mathfrak{P}_{\calB}\right) \cup \{\{1\},\dots,\{n\}\},
\end{align*}
i.e., non-singleton $\calA$-totally unstable partitions which are not also $\calB$-totally unstable.
The proof of this more general statement is essentially identical to the ones given above, but the notation becomes significantly more tedious. One benefit of the generalized statements is that they exhibit the $\psi$-class relations as a wall-crossing phenomenon in $((0,1]\cap\mathbb{Q})^n$, the parameter space for weight data.
\end{remark}

\section{Hassett Potentials and Operators}\label{sec:potential}

A common way to encode intersection data is via generating functions, or potentials. In this way recursions among intersection numbers and relations between different families of intersection numbers may be expressed via differential operators.

\begin{definition}\label{def:taunotation}
For any weight data $\A = (a_1,\dots,a_n)$ and genus $g$, define \emphbf{correlation functions} as
\begin{align}
\langle \tau_{a_1,k_1}\cdots\tau_{a_n,k_n}\rangle_{g} := \int _{\Moduli_{g,\A}} \prod_{i=1}^n  \psi_{i,\calA}^{k_i}.
\end{align}
In order to deal simultaneously with all intersection numbers on Hassett spaces with points having weights valued in a finite set $\underline{\calA} = \{a_1,\ldots,a_n\} $,
we adopt the following multi-index notation: given $\bm{b} = (b_{1,0},b_{2,0},\dots,b_{n,\ell})$
\begin{align}
\langle \bm{\tau^b}_{\underline\A} \rangle_g := \langle (\underbrace{\tau_{a_1,0}\cdots\tau_{a_1,0}}_{b_{1,0}\text{ factors}})\cdots (\underbrace{\tau_{a_n,0}\cdots\tau_{a_n,0}}_{b_{n,0}\text{ factors}}) \cdots  (\underbrace{\tau_{a_1,\ell}\cdots\tau_{a_1,\ell}}_{b_{1,\ell}\text{ factors}})\cdots (\underbrace{\tau_{a_n,\ell}\cdots\tau_{a_n,\ell}}_{b_{n,\ell}\text{ factors}}) \rangle_g.
\end{align}
If $2g - 2 + \sum a_i \leq 0$, we define the integral to be zero.
\end{definition}

\begin{definition}\label{def:hassettpotential}
Define a countable set of variables 
\begin{equation}
\bm{t} = \bigcup_{a \in (0,1]\cap\mathbb{Q}, k\in \mathbb{Z}^{\geq 0}}\{ t_{a,k}\};
\end{equation}
expanding the exponential function and using the same multi-index notation as above, we define the \emphbf{Hassett potential of genus $g$} to be
\begin{align} \label{hpg}
\calH_g(\bm{t}) := \left\langle exp\left(\sum_{a,k} t_{a,k}\tau_{a,k}\right) \right\rangle_g.
\end{align}
 The \emphbf{Hassett potential} is then
\begin{align}
\calH(\lambda;\bm{t}) := \sum_{g=0}^\infty \lambda^{2g-2}\calH_g(\bm{t}).
\end{align}
\end{definition}

Slightly more concretely, for any finite set $\underline \A = (a_1,\dots,a_n)\subset  ((0,1]\cap\mathbb{Q})^n$, denote
\begin{equation}
\bm{t}_{\underline\A} = \bigcup_{a_i\in \underline\calA, k\in \mathbb{Z}^{\geq 0}}\{ t_{a_i,k}\};
\end{equation}  then \eqref{hpg} means
\begin{align}
\calH_g(\bm{t}) = \sum_{\underline\A,\;\bm{b}}\frac{\bm{t^b}_{\A}}{\bm{b}!} \langle \bm{\tau^b}_{\underline\A} \rangle_g,
\end{align}
with $\bm{b}! = b_{1,0}!b_{1,1}!\cdots b_{n,\ell}!$.

\begin{example}
The monomial $\lambda^4t_{\frac{1}{2},5}\frac{t_{\frac{1}{3},2}^2}{2!}$ in $\calH(\lambda;\bm{t})$ has coefficient
\begin{align*}
\int_{\Moduli_{3,\calA}} \psi_{1,\calA}^5\psi_{2,\calA}^2\psi_{3,\calA}^2,
\end{align*}
where $\calA = \left(\frac{1}{2},\frac{1}{3},\frac{1}{3}\right)$ and $\underline\calA = \left(\frac{1}{2},\frac{1}{3} \right)$.
\end{example}





\begin{definition}\label{def:gwpotential}
Let $\bm{x} = (x_0,x_1,x_2,\dots)$ be an infinite sequence of formal variables. The \emphbf{Witten potential} or \emphbf{Gromov-Witten potential of a point}, denoted $\calF(\lambda;\bm{x})$ is the Hassett potential restricted to weight data of the form $\calA= 1^n$ for any $n$. Equivalently, $\calF(\lambda;\bm{x})$ is the Hassett potential with the substitution
\begin{align*}
t_{a,k} \mapsto \left\{\begin{array}{lr} 0, & \text{ if }a \neq 1 \\ x_k, & \text{ if } a=1\end{array}\right. 
\end{align*}
\end{definition}

 The Hassett potential is related to the Witten potential as prescribed by Corollary \ref{cor:numerical}. We now describe this connection in the language of generating functions via the following differential operator.

\begin{definition}\label{def:fork}
Let $\vardel_x := \frac{\vardel}{\vardel x}$. The \emphbf{fork operator} is
\begin{align}
\mathcal{L} &:= \sum_{n=1}^\infty \frac{(-1)^{n-1}}{n!}\sum_{\substack{(a_1,k_1),\dots,(a_n,k_n) \\ \sum a_j \leq 1}} t_{a_1,k_1}\cdots t_{a_n,k_n} \vardel_{x_{k_1+\cdots +k_n - (n-1)}}.
\end{align}
\end{definition}

The exponential of this operator relates relates the Witten potential to the Hassett potential.

\begin{theorem}\label{thm:wittentohassett}
Let $\mathcal{H}(\lambda;\bm{t})$, $\calF(\lambda;\bm{x})$, and $\calL$ be as above. Then
\begin{align} \label{eq:expop}
e^{\calL}\calF(\lambda;\bm{x})\big\vert_{\bm{x} = \bm{0}} = \mathcal{H}(\lambda;\bm{t}) + U,
\end{align}
where $U$ is supported on monomials $\prod_i t_{a_i,k_i}$  with $2g-2+\sum a_i \leq 0$.
\end{theorem}

\begin{proof}
This theorem is proved by checking that the equality of coefficients for each monomial in \eqref{eq:expop} follows from the formulas in Theorem \ref{thm:hassettcycle}. The details are in Appendix \ref{proofofwth}.
\end{proof}

\begin{remark}
When defining Hassett spaces in Section \ref{sec:hassettbackground}, we noted that $\Moduli_{g,\calA}$ is a non-empty, smooth, proper DM stack when $2g-2+\sum a_i > 0$. The coefficients appearing in $U$ ``would like to be'' intersection numbers on  Artin stacks of rational curves with weights whose sum is less than or equal to 2. We do not know if such an interpretation is possible.
\end{remark}

\begin{remark}\label{rem:generalwittentohassett}
In Remark \ref{rem:generalcycle} we noted that Corollary \ref{cor:numerical} generalizes in such a way as to explicitly connect two Hassett spaces of $\Moduli_{g,\calB}\to\Moduli_{g,\calA}$, without requiring $\calB = 1^n$. In complete analogy, the following statement can be seen as a generalization of Theorem \ref{thm:wittentohassett}.
Define
\begin{align}\label{def:generalfork}
\tilde{\calL} &:= \sum_{b \geq a} t_{a,i}\vardel x_{b,i} + \sum_{n\geq 2}^\infty \frac{(-1)^{n-1}}{n!}\sum_{\substack{(a_1,i_1),\dots,(a_n,i_n) \\ \sum a_j \leq 1}} t_{a_1,i_1}\cdots t_{a_n,i_n} \vardel x_{1,i_1+\cdots +i_n - (n-1)}.
\end{align}
Then
\begin{align}\label{generalexpop}
e^{\tilde{\calL}}\mathcal{H}(\lambda;{\bm{x}})\big\vert_{{\bm{x}} = \bm{0}} = \mathcal{H}(\lambda;\bm{t}) + U.
\end{align}
\end{remark}

\section{Changes of Variables}\label{sec:cov}

Observing that  $\calL$ is a vector field,  $e^{\calL}$ describes a flow along that vector field. Thus we have the following lemma.

\begin{lemma} \label{lem:expflow}
Rewrite the fork operator as
\begin{align} \label{eq:altoperator}
\mathcal{L} &= \sum_{i=0}^\infty f_i(\bm{t})\; \vardel_{x_{i}} \nonumber \\
&= \bm{f}(\bm{t})\; \vardel_{\bm{x}}.
\end{align}
Then
\begin{align} \label{eq:expflow}
\calH(\lambda;\bm{t}) + U &= e^{\mathcal{L}}\calF(\lambda;\bm{x})\big\vert_{\bm{x} = \bm{0}} \nonumber \\ \nonumber
&= \calF(\lambda;\bm{x} + \bm{f}(\bm{t}))\big\vert_{\bm{x} = \bm{0}} \\
&= \calF(\lambda; \bm{f}(\bm{t})).
\end{align}
\hfill $\square$
\end{lemma}

The Hassett potential is thus obtained from the Witten potential via a (non invertible) transformation of variables. 
While the general change of variables is rather cumbersome, if we narrow our attention to Hassett spaces of fixed diagonal weights, it may be made very explicit. We first define these restricted potentials.

\begin{definition}\label{def:qhpot}
For $q$ a fixed natural number, the \emphbf{$q$-diagonal Hassett potential} $\calD_q(\lambda;\bm{t})$ is the Hassett potential restricted to weight data of the form $\A = (\frac{1}{q},\ldots, \frac{1}{q})$. Equivalently, $\calD_q(\lambda;\bm{t})$ is the Hassett potential with all variables not of the form $t_{\frac{1}{q},k}$ set equal to zero.
\end{definition}

\begin{remark}
We note that for any rational number $a \in (0,1]$, the weight data $\calA = (a, \ldots, a)$ belongs to the same Hassett chamber as $(\frac{1}{q}, \ldots \frac{1}{q})$, where $q = \left\lfloor \frac{1}{a} \right\rfloor$. Hence the  functions $\calD_q$ are describing potentials for all possible choices of diagonal weights.
\end{remark}


For the sake of clarity, we restate some earlier results in the case of the diagonal potentials. Proofs of these statements are essentially identical to their earlier forms with appropriate notation changes; throughout, we assume a fixed $q\in\mathbb{N}$. First, a modification of Definition \ref{def:fork}.

\begin{definition}\label{def:forkdiag}
Let $\vardel_x := \frac{\vardel}{\vardel x}$. The \emphbf{$q$-fork operator} is
\begin{align}
\mathcal{L}_q &:= \sum_{n=1}^q \frac{(-1)^{n-1}}{n!}\sum_{i_1,\dots,i_n} t_{\frac{1}{q},i_1}\cdots t_{\frac{1}{q},i_n} \vardel_{x_{i_1+\cdots +i_n - (n-1)}}.
\end{align}
\end{definition}

Next, a diagonal version of Theorem \ref{thm:wittentohassett}.

\begin{theorem}\label{thm:wittentodiag}
Let $\mathcal{D}_q(\lambda;\bm{t})$, $\calF(\lambda;\bm{x})$, and $\calL_q$ be as above. Then
\begin{align} \label{eq:expopdiag}
e^{\calL_q}\calF(\lambda;\bm{x})\big\vert_{\bm{x} = \bm{0}} = \mathcal{D}_q(\lambda;\bm{t}) + U_q,
\end{align}
where $U_q$ is a collection of terms coming from ``unstable monomials": monomials of degree $n$ where $2g-2+ \frac{n}{q} \leq 0$. \hfill $\square$
\end{theorem}

Finally, the exponential flow for diagonal potentials, mirroring Lemma \ref{lem:expflow}.

\begin{lemma}\label{lem:expflowdiag}
Rewrite the $q$-fork operator as
\begin{align} \label{eq:altoperatordiag}
\mathcal{L}_q &= \sum_{i=0}^\infty g_i(\bm{t})\; \vardel x_{i} = \bm{g}(\bm{t})\; \vardel\bm{x}.
\end{align}
Then
\begin{align} \label{eq:expflowdiag}
\calD_q(\lambda;\bm{t}) + U_q &= e^{\mathcal{L}}\calF(\lambda;\bm{x})\big\vert_{\bm{x} = \bm{0}} \nonumber \\ \nonumber
&= \calF(\lambda;\bm{x} + \bm{g}(\bm{t}))\big\vert_{\bm{x} = \bm{0}} \\
&= \calF(\lambda; \bm{g}(\bm{t})).
\end{align}
\hfill $\square$
\end{lemma}

We now describe the change of variables $\bm{x}\mapsto \bm{g}(\bm{t})$.

\begin{corollary} \label{cor:changeofvariablesdiag}
The potential $\calD_q(\lambda;\bm{t})$ is obtained from $\calF(\lambda;\bm{x})$ via the change of variables given by the following equality of generating functions:
\begin{align}\label{eq:changeofvariablesdiag}
\sum_{i=0}^\infty x_{i}z^i = \left[z\left(\sum_{n = 1}^q \frac{1}{n!}\left(\sum_{k=0}^\infty -t_{k}z^{k-1} \right)^n \right)\right]_+,
\end{align}
 where we have denoted $t_{k}:= t_{\frac{1}{q},k}$ and  the subscript $+$ denotes the truncation of the expression in brackets to terms with non-negative exponents of  $z$.
\end{corollary}
\begin{proof}
Carefully examining the coefficients of $\mathcal{L}_q$ when written as in \eqref{eq:altoperatordiag}, one obtains
\begin{align*}
x_{0} = g_0(\mathbf{t}) &= t_0 -t_0t_1 +\frac{t_0^2}{2!}t_2 +t_0\frac{t_1^2}{2!} -\frac{t_0^3}{3!}t_3 - \frac{t_0^2}{2!}t_1 t_2 - t_0\frac{t_1^3}{3!} + \cdots \\
x_{1} = g_1(\mathbf{t}) &= t_1  -{t_0}t_2 - \frac{t_1^2}{2!} +\frac{t_0^2}{2!}t_3 +{t_0}t_1 t_2 + \frac{t_1^3}{3!} - \frac{t_0^3}{3!}t_4- \cdots \\
x_{2} = g_2(\mathbf{t}) &= t_2 - {t_0}t_3 -t_1 t_2  + \frac{t_0^2}{2!}t_4+ t_0 t_1t_3 + t_0\frac{t_2^2}{2!}+\frac{t_1^2}{2!}t_2-  \cdots \\
x_{3} = g_3(\mathbf{t}) &= t_3   - {t_0}t_4-  t_1t_3 - \frac{t_2^2}{2!} + \frac{t_0^2}{2!}t_5 + t_0t_1t_4+ t_0 t_2t_3+ \frac{t_1^2}{2!}t_3+t_1 \frac{t_2^2}{2!}-  \cdots \\
\end{align*}
These sums are finite: only monomials of total degree less than or equal to $q$ appear. It is a simple combinatorial exercise to see that the above change of variables is organized in generating function form as in \eqref{eq:changeofvariablesdiag}.
\end{proof}

\begin{remark}
Corollary \ref{cor:changeofvariablesdiag} gives a complete picture of the wall-crossings for diagonal Hassett potentials. We remark that the case $q = 2$ was studied in \cite{nand} and the asymptotic case $q\to +\infty$ appears in \cite{bckappa}.
\end{remark}

\section{Cycle-Valued Potentials}\label{sec:cyclegen}

In this section we generalize Theorem \ref{thm:wittentohassett} to the level of cycles, showing that the potential for all monomials of $\psi$ classes on all Hassett spaces can be related to a potential for cycles on Deligne-Mumford spaces $\overline{\mathcal{M}}_{g,n}$ obtained by decorating pinwheel strata with $\psi$ classes.

\begin{definition}\label{def:altwitten}
Let $\mathbb{H}$ be a countably dimensional vector space, with  coordinates $x_0, x_1, x_2, \ldots$. These variables can be graded in two different, meaningful ways: we say that $x_k$ has numerical degree equal to one and Chow degree equal to $k$. Then the \emphbf{genus $g$ cycle Witten potential} is the function
\begin{align}
{\cycF}_g:\mathbb{H} \to \prod_{g,n} A^*(\Moduli_{g,n})/S_n
\end{align}
defined by
\begin{align}\label{cgwp}
\bm{x}=(x_0,x_1,\dots) \mapsto \exp\left(\sum_{k=0}^\infty x_k\psi^k\right),
\end{align}
where \eqref{cgwp} is shorthand notation to denote the fact that the coefficient of a monomial of numerical degree $n$ is a cycle on $\Moduli_{g,n}$. The Chow degree of the monomial equals the Chow degree of the cycle. The natural action of the symmetric group $S_n$ on $\Moduli_{g,n}$ permuting the points induces an action on the Chow ring: the cycle potential takes value in the quotient via such action, which means that we do not keep track of the actual labels of the marked points, but only of the number of $\psi$ classes raised to any given power.
As usual we denote the \emphbf{cycle Witten potential} by summing over genus
\begin{align}
{\cycF}:= \sum_{g=0}^{\infty} \lambda^{2g-2} {\cycF}_g.
\end{align}
\end{definition}

We now generalize the cycle Witten potential in two ways, first to encode cycles supported on pinwheel strata, then to encode monomials of $\psi$ classes pulled-back from Hassett spaces.

\begin{definition}\label{def:cycpin}
Let $\mathbb{V}$ be a countably dimensional vector space, with  coordinates $x_{n,k}$, where $n\geq 1$ and $k\geq 0$. The variable $x_{n,k}$ has numerical degree equal to $n$ and Chow degree equal to $k$. 
For $g\geq 1$ and $n\geq 2$, define the gluing morphism
\begin{align}
gl^{n}:\Moduli_{g,\{\bullet\}}\times \Moduli_{0,n\cup\{\star\}}\to\Moduli_{g,n}
\end{align}
which glues together $\bullet$ and $\star$. 
Then the \emphbf{genus $g$ pinwheel potential} is the function
\begin{align}
{\cycQ}_g:\mathbb{V} \to \prod_{g,n} A^*(\Moduli_{g,n})/S_n
\end{align}
defined by
\begin{align}\label{cpp}
\bm{x}=(x_{n,k}) \mapsto \exp\left(\sum_{k=0}^\infty x_{1,k}\psi^i- \sum_{n=2}^\infty x_{n,k} gl_*^{n}\left(\frac{\psi_{\bullet}^{k-1} - (- \psi_{\star})^{k-1}}{\psi_\bullet + \psi_{\star}}\right)\right),
\end{align}
where \eqref{cpp} is shorthand notation to denote the fact that the coefficient of a monomial of numerical degree $n$ is a cycle on $\Moduli_{g,n}$, via the natural generalization of gluing morphisms so that multiple tails are attached. The cohomological degree of the monomial equals the cohomological degree of the cycle.
The \emphbf{cycle pinwheel potential}  is obtained by summing over genus in the usual fashion.
\end{definition}

\begin{definition}\label{def:cychas}
Let $\mathbb{W}$ be a countably dimensional vector space, with  coordinates $t_{a,k}$, with $a\in (0,1]\cap \mathbb{Q}$ and $k\geq 0$.  The variable $t_{a,k}$ has numerical degree equal to $1$ and cohomological degree equal to $k$. We also say that the variable $t_{a,k}$ has weight $a$.
Let $c_a: \Moduli_{g,1} \to \Moduli_{g, \{a\}}$ denote the contraction morphism to a one pointed Hassett space where the point has weight $a$.
 Then the \emphbf{genus $g$ cycle Hassett potential} is the function
\begin{align}
{\cycH}_g:\mathbb{W} \to \prod_{g,n} A^*(\Moduli_{g,n})/S_n
\end{align}
defined by
\begin{align}\label{chp}
\bm{t}=(t_{a,k}) \mapsto \exp\left(\sum_{a}\sum_{k=0}^\infty t_{a,k}c_a^\ast\psi^k\right),
\end{align}
where \eqref{chp} is shorthand notation to denote the fact that the coefficient of a monomial of numerical degree $n$ is a cycle on $\Moduli_{g,n}$, pulled-back via an appropriate contraction morphism $c_{\calA}$ determined by the weights of the variables in the monomial. The cohomological degree of the monomial equals the cohomological degree of the cycle.
The \emphbf{cycle Hassett potential} is obtained by summing over genus.
\end{definition}

With these definitions in place the shape of the fork operator and the statement of Theorem \ref{thm:wittentohassett} for cycle potentials are natural.

\begin{definition}\label{def:cfork} The \emphbf{cycle fork operator} is
\begin{align}
{\cycL} &:= \sum_{n=1}^\infty \frac{1}{n!}\sum_{\substack{(a_1,k_1),\dots,(a_n,k_n) \\ \sum a_j \leq 1}} t_{a_1,k_1}\cdots t_{a_n,k_n} \vardel_{x_{n, k_1+\cdots +k_n}}.
\end{align}
\end{definition}

\begin{theorem}\label{thm:wittentohassettcyc}
Let $\mathcal{H}(\lambda;\bm{t})$, $\cycQ(\lambda;\bm{x})$, and $\cycL$ be as above. Then
\begin{align} \label{eq:expopc}
e^{\cycL}\cycQ(\lambda;\bm{x})\big\vert_{\bm{x} = \bm{0}} ={\cycH}(\lambda;\bm{t}) + U,
\end{align}
where $U$ is supported on monomials $\prod_i t_{a_i,k_i}$  with $2g-2+\sum a_i \leq 0$.
\end{theorem}

\begin{proof}
The proof consists of comparing the coefficients of each monomial on both sides of \eqref{eq:expopc} and observing that their equality is exactly the statement of Theorem \ref{thm:hassettcycle}. 
\end{proof}

\appendix
\section{Proofs of technical theorems}
\subsection{Theorem \ref{thm:hassettcycle}} \label{hacyth}
\begin{proof}
We adopt the convention $\psi_{\star_j} = 0$ when $|P_j| = 1$. The proof consists of an induction on $n$ and the total power $K =\sum_{i=1}^n k_i$. 
We proceed by induction on $(n,K)$ in lexical order: we assume the formula holds for all pairs $(n, K)$ with $n< n_0$ or $n=n_0$ and $K\leq K_0$, and prove it for $(n_0, K_0+1)$.
 
We begin with  the base case $K=1$ for every $n$. On the left-hand side of \eqref{eq:hassettfor}, we have, without loss of generality, $c^*_{\A}\psi_{1,\calA}$. On the right-hand side, the  pinwheel stratum  corresponding to the singleton partition $\calP = \{\{1\},\{2\},\dots,\{n\}\}$ is $\Moduli_{g,n}$; it appears in \eqref{eq:hassettfor} with coefficient  $\psi_1$. 
As seen in Remark \ref{rem:ratfun}, non-zero contributions only come from strata where each part of size greater than one has a point with non-zero $k_1$. In this case, this only leaves partitions with exactly one part $P_j$ of size greater than one; further, it must be that $1\in P_j$ and $\calP \in \mathfrak{P}_{\A}$. We have $[\Delta_\calP] = D([n]\smallsetminus P_j\,|\,P_j)$, $\alpha_j = 1$, and all other $\alpha$ equal to zero.  The coefficient from $P_j$ is
\begin{align*}
\frac{\psi_{\bullet_j}-(-\psi_{\star_j})}{-\psi_{\bullet_j}-\psi_{\star_j}} = -1.
\end{align*}
All other parts are singletons with $\alpha=0$, and hence  each contributes $\psi_{\bullet}^0 = 1$ to the product. 
Thus equation \eqref{eq:hassettfor} becomes 
\begin{align}
c^*_{\A}\psi_{1,\calA} 
&= \psi_1 - \sum_{\substack{1\in B \\ D(A\,|\,B)\in\mathfrak{P}_{\A}}} D(A\,|\,B),
\end{align}
which we have seen in Lemma \ref{lem:pull-back}. Thus the base case is established.

Assume \eqref{eq:hassettfor} holds for total monomial power $K \leq m$ for some $m\in\mathbb{N}$ and for all spaces with fewer than $n$ marked points. We hold $n$ fixed and increase $K$ by 1 by multiplying, again without loss of generality, by $c^*_{\A}\psi_{1,\calA}$. We have
\begin{align} \label{for:ind}
c^*_{\A}\left(\prod_{i=1}^n \psi_{i,\calA}^{k_i}\right)\cdot c^*_{\A}\psi_{1,\calA} &=  c^*_{\A} \left(\prod_{i=1}^n \psi_{i,\calA}^{k_i}\right) \left(\psi_1 - \sum_{\substack{1\in B \\ D(A\,|\,B) \in \mathfrak{P}_{\A}}} D(A|B) \right) \nonumber \\
&= c^*_{\A}\left(\prod_{i=1}^n \psi_{i,\calA}^{k_i}\right)\cdot \psi_1 - \sum_{\substack{1\in B \\ D(A\,|\,B)\in \mathfrak{P}_{\A}}} c^*_{\A}\left(\prod_{i=1}^n \psi_{i,\calA}^{k_i}\right)  D(A\,|\,B). 
\end{align}
We may  assume $1\in P_1$. We examine each of the summands on the right-hand side of \eqref{for:ind}. For the first term, by inductive hypothesis we have
\begin{align}
c^*_{\A}\left(\prod_{i=1}^n \psi_{i,\calA}^{k_i}\right)\cdot\psi_1 &= \left(\sum_{\calP \in \mathfrak{P}_{\A}} [\Delta_{\calP}] \prod_{|P_j|=1} \psi_{\bullet_j}^{\alpha_j} \prod_{|P_j|>1} \decoratenew \right)\cdot \psi_1 \nonumber \\
&= \sum_{|P_1|=1} [\Delta_{\calP}] \prod_{\substack{|P_j|=1 \\ j\neq 1}} \psi_{\bullet_j}^{\alpha_j} \prod_{|P_j|>1} \decoratenew \cdot \psi_1^{k_1+1}  \nonumber \\
&\hspace{1cm} + \sum_{|P_1|>1} \left([\Delta_{\calP}]\cdot \psi_1\right)\prod_{|P_j|=1} \psi_{\bullet_j}^{\alpha_j} \prod_{|P_j|>1} \decoratenew. \label{eq:lhs}
\end{align}
Note  that for the $|P_1|=1$ cases, $\{\bullet_1\} = \{1\}$ by the  convention adopted in defining $[\Delta_{\calP}]$.

We now turn to the second summand in \eqref{for:ind}. We rename $B = P_1$ to emphasize that  the point $1$  belongs to this subset; now $A = [n]\backslash P_1 = \{i_1,\dots,i_\ell\}$ (which may be empty) 
and note that  summing over all divisors $D(A\,|\,B)$ with $1\in B$ is equivalent to summing over all 
$|P_1|>1$ by the stability requirement on rational components. We denote $gl_{P_1}: \Moduli_{g, A\cup \{L_1\}} \times \Moduli_{0,P_1\cup\{R_1\}} \to \Moduli_{g,n}$ the gluing morphism whose image is $D(A\,|\,P_1)$. Let $s = \min\left(1,\sum_{i\in P_1} a_i\right)$; then we have:
\begin{align}
\hspace{-1cm}\sum_{\substack{1\in B \\ D(A\,|\,B)\in\mathfrak{P}_{\A}}} c^*_{\A}\left(\prod_{i=1}^n \psi_{i,\calA}^{k_i}\right)  & D(A\,|\,B) = \sum_{|P_1|>1} c^*_{\A}\left(\prod_{i=1}^n \psi_{i,\calA}^{k_i}\right)  D(A\,|\,P_1) \nonumber \\
&\stackrel{Lemma\ \ref{lem:falldownhassett}}{=} \sum_{|P_1|>1} {gl_{P_1}}_*\left[c^*_{\calA}\left(\prod_{i\in A}\psi_{i,\calA}^{k_i}\right) \cdot c^*_{\calB}\left(\psi_{L_1,\calB}^{\sum_{j\in P_1} k_j} \right)\right]  \nonumber\\
&= \sum_{|P_1|>1} {gl_{P_1}}_*\left(\sum_{\calQ\in\mathfrak{P}_{\calB}} [\Delta_{\calQ}] \prod_{|Q_j|=1} \psi_{\bullet_j}^{\alpha_j} \prod_{|Q_j|>1} \decoratenew  \right)  \label{eq:dc}
\end{align}
where $\calB = (a_{i_1},\dots,a_{i_{\ell}},a_{L_1} = s)$.


The last equality follows from induction with respect to the number of marks. Adopting the convention that $L_1 \in Q_1$, we note that $\alpha_1 = \sum_{i\in Q_1\cup P_1}k_i$. We now group the partitions $\calQ$ in two groups: the first is where $Q_1$ is the singleton $\{L_1\}$: in this case ${gl_{P_1}}_\ast ([\Delta_\calQ])$ is the class of the pinwheel stratum $\Delta_\calP$, where $\calP = P_1 \cup \calQ \smallsetminus Q_1$. The second group contains all partitions $\calQ$ with $|Q_1|>1$. See Figure \ref{fig-twogps} for a pictorial description. Then \eqref{eq:dc} continues:

\begin{align}
&= \sum_{|P_1|>1} [\Delta_{\calP}] \prod_{|P_j| = 1} \psi_{\bullet_j}^{\alpha_j} \prod_{\substack{|P_j|>1 \\ j\neq 1}} \decoratenew \cdot \psi_{\bullet_1}^{\alpha_1}  \nonumber \\
&\hspace{1cm} + \sum_{|P_1|>1}\sum_{|Q_1|>1} {gl_{P_1}}_*\left( [\Delta_{\calQ}] \cdot \decorateonenew \prod_{|Q_j|=1} \psi_{\bullet_j}^{\alpha_j} \prod_{\substack{|Q_j|>1 \\ j\neq 1}} \decoratenew  \right) \nonumber \\ 
&= \sum_{|P_1|>1} [\Delta_{\calP}] \prod_{|P_j| = 1} \psi_{\bullet_j}^{\alpha_j} \prod_{\substack{|P_j|>1 \\ j\neq 1}} \decoratenew \cdot \psi_{\bullet_1}^{\alpha_1}  \nonumber \\
&\hspace{1cm} + \sum_{|P_1|>1} [\Delta_{\calP}] \prod_{|P_j|=1} \psi_{\bullet_j}^{\alpha_j} \prod_{|P_j|>1} \decoratenew \cdot (\psi_1+\psi_{\star_1}) 
\label{eq:rhs}
\end{align}

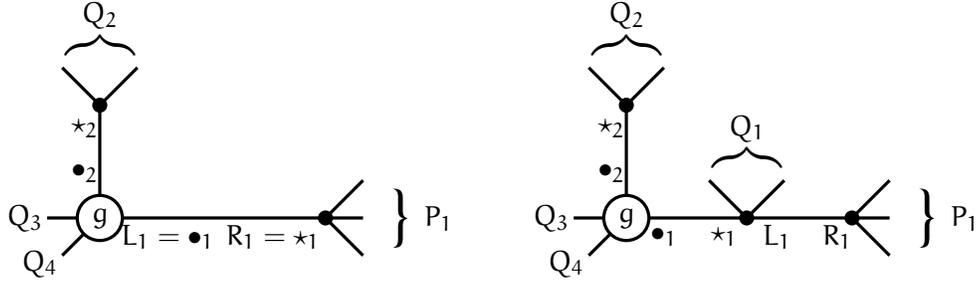
\begin{figure}[bt]

\begin{tikzpicture}
\draw[very thick]  (0,0) circle (0.30);
\node at (0,0) {$g$};

\draw[very thick] (0.30,0) -- (3,0);
\fill (3,0) circle (0.10);
\draw[very thick] (3,0) -- (3.5,0);
\draw[very thick] (3,0) -- (3.5,0.5);
\draw[very thick] (3,0) -- (3.5,-0.5);
\node at (4,0) {{\Huge{\}}}};
\node at (4.5,0) {$P_1$};
\node at (0.9,-0.25) {$L_1 = \bullet_1$};
\node at (2.3,-0.25) {$R_1= \star_1$};

\draw[very thick] (0,0.30) -- (0,1.5);
\fill (0,1.5) circle (0.10);
\draw[very thick] (0,1.5) -- (-0.5,2);
\draw[very thick] (0,1.5) -- (0.5,2);
\node at (0,2.3) {{\Large$\overbrace{}$}};
\node at (0,2.7) {$Q_2$};

\node at (-0.2, .6) {$\bullet_2$};
\node at (-0.2, 1.2) {$\star_2$};

\draw[very thick] (-0.30,0) -- (-.7,0);
\node at (-1,0) {$Q_3$};

\draw[very thick] (-0.21,-0.21) -- (-.5,-.5);
\node at (-0.8,-.6) {$Q_4$};
\draw[very thick]  (7,0) circle (0.30);
\node at (7,0) {$g$};

\draw[very thick] (7.30,0) -- (10,0);
\fill (10,0) circle (0.10);
\draw[very thick] (10,0) -- (10.5,0);
\draw[very thick] (10,0) -- (10.5,0.5);
\draw[very thick] (10,0) -- (10.5,-0.5);
\node at (11,0) {{\Huge{\}}}};
\node at (11.5,0) {$P_1$};
\node at (7.5,-0.25) {$\bullet_1$};
\node at (9,-0.25) {$L_1$};
\node at (8.3,-0.25) {$\star_1$};
\node at (9.8,-0.25) {$R_1$};

\fill (8.6,0) circle (0.10);
\draw[very thick] (8.6,0) -- (8.1,0.5);
\draw[very thick] (8.6,0) -- (9.1,0.5);
\node at (8.6,0.8) {{\Large$\overbrace{}$}};
\node at (8.6,1.2) {$Q_1$};

\draw[very thick] (7,0.30) -- (7,1.5);
\fill (7,1.5) circle (0.10);
\draw[very thick] (7,1.5) -- (6.5,2);
\draw[very thick] (7,1.5) -- (7.5,2);
\node at (7,2.3) {{\Large$\overbrace{}$}};
\node at (7,2.7) {$Q_2$};

\node at (6.8, .6) {$\bullet_2$};
\node at (6.8, 1.2) {$\star_2$};

\draw[very thick] (6.7,0) -- (6.3,0);
\node at (6,0) {$Q_3$};

\draw[very thick] (7-0.21,-0.21) -- (7-.5,-.5);
\node at (7-0.8,-.6) {$Q_4$};

\end{tikzpicture}
\caption{Examples of dual graphs corresponding to the two summands in equation \eqref{eq:rhs}. On the left-hand side we have graphs where $Q_1 = \{L_1\}$; on the right-hand side $|Q_1| >1$. }\label{fig-twogps}
\end{figure}

In the last equality we have applied Lemma \ref{lem:psisums}, and reindexed the sum so that the new $P_1$ is equal to what used to be $P_1 \cup (Q_1 \smallsetminus L_1)$.


We rewrite \eqref{for:ind} using \eqref{eq:lhs} and \eqref{eq:rhs};  the second term in the right-hand side of \eqref{eq:lhs}  cancels part of the  second term of the right-hand side of \eqref{eq:rhs}, and we obtain:

\begin{align}
\label{almostthere}
c^*_{\A}\left(\prod_{i=1}^n \psi_{i,\calA}^{k_i}\right)\cdot c^*_{\calA}\psi_{1,\calA} &= \sum_{|P_1|=1} [\Delta_{\calP}] \prod_{\substack{|P_j|=1 \\ j\neq 1}} \psi_{\bullet_j}^{\alpha_j} \prod_{|P_j|>1} \decoratenew \cdot \psi_1^{k_1+1}  \nonumber \\
&\hspace{1cm} -\sum_{|P_1|>1} [\Delta_{\calP}] \prod_{|P_j| = 1} \psi_{\bullet_j}^{\alpha_j} \prod_{\substack{|P_j|>1 \\ j\neq 1}} \decoratenew \cdot \psi_{\bullet_1}^{\alpha_1}  \nonumber \\
&\hspace{1cm} - \sum_{|P_1|>1} [\Delta_{\calP}] \prod_{|P_j|=1} \psi_{\bullet_j}^{\alpha_j} \prod_{|P_j|>1} \decoratenew \cdot \psi_{\star_1} \nonumber \\
&= \sum_{|P_1|=1} [\Delta_{\calP}] \prod_{\substack{|P_j|=1 \\ j\neq 1}} \psi_{\bullet_j}^{\alpha_j} \prod_{|P_j|>1} \decoratenew \cdot \psi_1^{k_1+1} \nonumber \\
&\hspace{1cm}+ \sum_{|P_1|>1} [\Delta_{\calP}] \prod_{|P_j| = 1} \psi_{\bullet_j}^{\alpha_j} \prod_{\substack{|P_j|>1 \\ j\neq 1}} \decoratenew \nonumber \\
&\hspace{2cm} \times \left(-\psi_{\bullet_1}^{\alpha_1}  -\psi_{\star_1}\decorateonenew \right)
\end{align}

We conclude the proof by observing that \eqref{almostthere} gives formula \eqref{eq:hassettfor}, with $k_1$ replaced by $k_1+1$ (and hence every occurrence of $\alpha_1$ replaced by $\alpha_1+1$): the first summand shows that the coefficients agree on the nose for partitions with $P_1 = \{1\}$; the second summand deals with partitions where $|P_1|>1$: the coefficients match after noting the elementary identity:
\begin{align*}
\frac{\psi_{\bullet_1}^{\alpha_1+1}-(-\psi_{\star_1})^{\alpha_1+1}}{-\psi_{\bullet_1}-\psi_{\star_1}} &= -\psi_{\bullet_1}^{\alpha_1} + \psi_{\bullet_1}^{\alpha_1-1}\psi_{\star_1} - \psi_{\bullet_1}^{\alpha_1-2}\psi_{\star_1}^2 + \cdots \\
&= -\psi_{\bullet_1}^{\alpha_1} - \psi_{\star_1}\cdot (-\psi_{\bullet_1}^{\alpha_1-1} + \psi_{\bullet_1}^{\alpha_1-2}\psi_{\star_1} - \cdots). 
\end{align*}
\end{proof}

\subsection{Theorem \ref{thm:wittentohassett}}
\label{proofofwth}

\begin{proof}
The strategy is to use \eqref{eq:numerical} and careful bookkeeping to show that the coefficients of both sides of \eqref{eq:expop} agree. We denote the multi-index
\begin{align*}
K = (1^{k_1}, 2^{k_2}, \dots, m^{k_m}) = (\underbrace{1,\dots, 1}_{k_1\text{ factors}}, \,\underbrace{2, \dots, 2}_{k_2\text{ factors}}, \dots, \underbrace{m, \dots, m}_{k_m \text{ factors}});
\end{align*}
all multi-indices arising in this proof may be assumed to be normalized so that the values are non-decreasing. 

Fix a monomial in $\mathcal{H}(\lambda;\bm{t})$
\begin{align*}
\lambda^{2g-2}\frac{t_\calA^K}{K!}:=\lambda^{2g-2}\frac{\prod t_{a,m}^{k_{a,i}}}{\prod k_{a,i}!} .
\end{align*}
The coefficient $\mathcal{I}$ of this monomial is the  intersection number described in Definition \ref{def:hassettpotential}. Formula  \eqref{eq:numerical} gives an expression for this quantity in terms of a weighted sum over all $\A$-totally unstable partitions of the index set, so let us fix a partition $\calP = \{P_1, \dots , P_r\} \in \mathfrak{P}_{\A}$. Each part $P_i$ of the partition $\calP$ produces a multi-index by looking at the powers of weighted $\psi$-classes supported on the points that belong to $P_i$.

It is possible that different parts give rise to the same multi-index. We denote $\underline{J} = (J_1^{\nu_1}, \dots, J_t^{\nu_t})$ the collection of the multi-indices arising from the parts of $\calP$, intending that the multi-index $J_i$ arises $\nu_i$ times. For  $i$ from $1$ to $t$, we denote  $J_i = (1^{j_{i,1}}, \ldots, m^{j_{i,m}})$. Finally, from $\underline{J}$, we can produce a multi-index $\alpha = (\alpha_1^{\nu_1}, \ldots,\alpha_t^{\nu_t})$, where $\alpha_i = 1 +  j_{i,2}+2j_{i,3}+\ldots+(m-1) j_{i,m}$. 
The multi-index $\alpha$ is the exponent vector of the monomial in $\psi$-classes corresponding to the partition $\calP$ in \eqref{eq:numerical}.
There are exactly
\begin{align*}
\frac{\prod k_{a,i}!}{((\prod j_{i,l}!)^{\nu_i}\prod \nu_i!)}
\end{align*}
distinct partitions of the indices that will produce a given multi-index $\alpha$. With this notation established, we can rewrite \eqref{eq:numerical} as a summation over the combinatorial data given by $\underline{J}$:
\begin{align} \label{eq:ugh}
\mathcal{I} = \sum_{\underline{J}} (-1)^{n+\ell(\alpha)} \frac{\prod k_{a,i}!}{(\prod j_{i,l}!)^{\nu_i}\prod \nu_i!}\int_{\Moduli_{g,\ell(\alpha)}}\psi^\alpha.
\end{align}
We exhibit a summand $m_\alpha$ in $\calF(\lambda;\bm{x})$, and a term $L_{\underline{J}}$ in the differential operator $e^{\calL}$ such that $L_{\underline{J}}m_\alpha$ is a multiple of $\lambda^{2g-2}\dfrac{t_\calA^{K}}{K!}$; the specific multiple may be shown to be precisely $\dfrac{\prod k_{a,i}!}{((\prod j_{i,l}!)^{\nu_i}\prod \nu_i!)}$. Define:
\begin{align}
m_\alpha & =  \left(\int_{\Moduli_{g,\ell(\alpha)}}\psi^\alpha\right) \lambda^{2g-2} \frac{x_{\alpha_1}^{\nu_1}\cdots x_{\alpha_t}^{\nu_t}}{\nu_1!\cdots \nu_t!},\\
L_{\underline{J}} & =\frac{ (-1)^{n+\ell(\alpha)} }{\prod \nu_i!} \prod \left( \frac{\prod t_{a,i}^{j_{a,i}}}{\prod j_{a,i}!} \partial{x_{\alpha_i}}\right)^{\nu_i}.
\end{align}
One may show that all pairs of terms  $m$ in $\calF(\lambda;\bm{x})$ and $L$ in $e^{\calL}$ such that $Lm$ is a multiple  of $\lambda^{2g-2}\dfrac{t_\calA^{K}}{K!}$ arise in this fashion. It follows that the coefficient of $\lambda^{2g-2}\dfrac{t_\calA^{K}}{K!}$ in $e^{\calL} \calF(\lambda;\bm{x})$ equals the right hand side of \eqref{eq:ugh}, and therefore it agrees with the coefficient of the same monomial $\calH(\lambda;\bm{t})$.
\end{proof}

\bibliographystyle{amsalpha} 
\bibliography{bibliography} 

\end{document}